\documentclass[12pt]{article}
\usepackage{amsmath,amssymb,amsthm}
\usepackage{amsfonts}

\usepackage{color}

\newtheorem{theo}{Theorem}[section]
\newtheorem{lema}[theo]{Lemma}
\newtheorem{coro}[theo]{Corollary}
\newtheorem{prop}[theo]{Proposition}
\theoremstyle{definition}
\newtheorem{defi}[theo]{Definition}

\def\mut{\widetilde{\mu}}
\def\H{\mathbb{H}}
\def\C{\mathbb{C}}
\def\N{\mathcal{N}}
\def\R{\mathbb{R}}
\def\Sc{\operatorname{Sc}}
\def\Vec{\operatorname{Vec}}

\def\arctanh{\operatorname{arctanh}}
\def\DS{\displaystyle}

\newcommand{\coef}[2]{#1_{#2}}
\newcommand{\cUU}[1]{\coef{u}{#1}}
\newcommand{\cVV}[1]{\coef{v}{#1}}
\newcommand{\cVVmu}[1]{\coef{v}{#1}}
\newcommand{\cUV}[1]{\coef{w}{#1}}

\newcommand{\cVUmu}[1]{\coef{t}{#1}}
\newcommand{\chypgeom}[1]{\coef{c}{#1}}

\newcommand{\cZZmu}[2]{\coef{z}{#1}^{\rm#2}}
 
\newcommand{\cscale}[1]{\coef{\alpha}{#1}}
\newcommand{\cmumut}[1]{\coef{b}{#1}}
\newcommand{\normrat}[1]{\coef{a}{#1}}
\newcommand{\normconst}[1]{\coef{\beta}{#1}}

\newcommand{\fun}[3]{#1_{#2}^{#3}}
\newcommand{\funhat}[3]{\fun{\widehat{#1}}{#2}{#3}}
\newcommand{\funbar}[3]{\fun{\overline{#1}}{#2}{#3}}
\newcommand{\harorig}[1]{\funhat{U}{#1}{}}      
\newcommand{\harorigsol}[2]{\fun{U}{#1}{#2}}    
\newcommand{\hargarab}[1]{\funhat{V}{#1}{}}     
\newcommand{\hargarabsol}[2]{\fun{V}{#1}{#2}}   
\newcommand{\monog}[2]{\fun{X}{#1}{#2}}
\newcommand{\antimonog}[2]{\funbar{X}{#1}{#2}}
\newcommand{\ambibasic}[2]{Y_{#1}^{#2}}
\newcommand{\contra}[2]{\fun{Z}{#1}{#2}}
\newcommand{\ambig}[2]{A_{#1}^{#2}}
\newcommand{\angPsi}[2]{\fun{\Psi}{#1}{#2}}     

\begin{document}

\begin{center}
 {\Large\bf Relations among spheroidal\\ and spherical harmonics}

\medskip
\medskip

Raybel Garc{\'i}a-Ancona\\
Jo\~{a}o Morais\\
R. Michael Porter

 \today
\end{center}

\begin{abstract}
  A \textit{contragenic} function in a domain $\Omega\subseteq\R^3$ is
  a reduced-quaternion-valued (i.e\ the last coordinate function is
  zero) harmonic function, which is orthogonal in
  $L_2(\Omega)$ to all monogenic functions and their conjugates.  The
  notion of contragenicity depends on the domain and thus is not a
  local property, in contrast to harmonicity and monogenicity.  For
  spheroidal domains of arbitrary eccentricity, we relate standard
  orthogonal bases of harmonic and contragenic functions for one
  domain to another via computational formulas.  This permits us to 
  show that there exist nontrivial contragenic functions common to the
  spheroids of all eccentricities.
\end{abstract}

\noindent \textbf{Keywords:} spherical harmonics, spheroidal
harmonics, quaternionic analysis, monogenic function, contragenic
function.

\medskip
\noindent \textbf{Classification:} 30G35; 33D50

\section{Introduction}

In certain physical problems in nonspherical domains, it has been
found convenient to replace the classical solid spherical harmonics
with harmonic functions better adapted to the domain in question.  For
example, spheroidal harmonics are used in \cite{Hot} for modeling
potential fields around the surface of the earth.

A systematic analysis of harmonic functions on spheroidal domains was
initiated by Szeg\"o \cite{Szego1935}, followed by Garabedian
\cite{Garabedian} who produced orthogonal bases with respect to
certain natural inner products associated to prolate and oblate
spheroids, among them the $L^2$-Hilbert space structures on the
interior and on the boundary of the spheroid. Some aspects of the
generation of harmonic functions which are orthogonal in the region
exterior to a prolate spheroid were considered in
\cite{MoraisNguyenKou2016} and generalized recently in
\cite{Morais_Habilitation2018}.

The main question which interests us is to relate systems of
harmonic functions associated with the spheroid $\Omega_\mu$ (defined
in \eqref{eq:Omegamu} below) to those associated with the unit ball
$\Omega_0$.
Our starting point is a fundamental formula for spheroidal harmonics
which was worked out in the short but beautiful paper \cite{BBS} and
is discussed thoroughly in Chapter 22 of the monumental text
\cite{Hot}. In classical books such as
\cite{Hobson1931,MorseFeshbach1953,NikiforovUvarov1988}, these
expansions in terms of these bases are used separately without
specifying relations between them.

We complete the above formulas by relating different systems of
harmonic functions associated with spheroids of different
eccentricity. While the manipulation of the coefficients is
essentially algebraic, it must be borne in mind that we are dealing
with continuously varying families of function spaces which are
determined by integration over varying domains.

This study is then extended to include the contragenic functions,
which are those harmonic functions orthogonal to both the
monogenic functions and the antimonogenic functions in the domain
under consideration. In \cite{GMP} a short table of contragenic
polynomials was provided, which included some which did not depend on
the parameter describing the eccentricity of the spheroid.  Such
polynomials are thus contragenic for all spheroids. Our main result,
Theorem \ref{th:intersection}, describes the
intersection of the spaces of contragenic functions.


\section{Background on spheroidal harmonics}
 
As a preliminary to the discussion of monogenic and contragenic
functions on spheroids, we establish the basic facts for harmonics in
this and the next section.  Consider the family of coaxial spheroidal
domains $\Omega_\mu$, scaled so that the major axis is of length 2:
\begin{equation}\label{eq:Omegamu}
  \Omega_\mu =
  \{x \in \R^3 |\  x_0^2 + \frac{x_1^2 + x_2^2}{e^{2\nu}} < 1 \},
\end{equation}
where $\nu\in\R$ and where following the notation in \cite{GMP} the
parameter $\mu=(1-e^{2\nu})^{1/2}$ will be useful in later formulas.
The equations relating the Cartesian coordinates of a point
$x=(x_0,x_1,x_2)$ in  $\Omega_\mu$ to spheroidal coordinates
$(u,v,\phi)$ are
\begin{equation}\label{eq:prolatecoords}
  x_0 = \mu \cos u \cosh v, \ x_1 = \mu \sin u \sinh v \cos \phi, \
  x_2 = \mu \sin u \sinh v \sin \phi,
\end{equation}
where in the case of the prolate spheroid ($\nu<0$) the coordinates
range over $u\in[0,\pi]$, $v\in[0,\arctanh e^\nu]$,
$\phi\in[0,2\pi)$ and the eccentricity is $0<\mu<1$, while for the
oblate spheroid ($\nu>0$) we have $u\,\in\,[0,\pi]$ and
$v\in[0,\arctanh e^\nu]$, $\phi\in[0,2\pi)$ and $\mu$ is imaginary,
$\mu/i>0$.  The spheroids reduce to the unit ball $\Omega_0$ for
$\nu=0$. In many other treatments of spheroidal functions, which
discuss the two (confocal) families separately, the ball is
not represented. See \cite{GMP} for a discussion of this question.

In terms of the coordinates \eqref{eq:prolatecoords}, the {\it solid
  spheroidal harmonics} are
\begin{equation} \label{eq:prolatepreharmonics}
  \harorigsol{n,m}{\pm}[\mu](x) = \harorig{n,m}[\mu](u,v) \, \Phi_m^\pm(\phi),
\end{equation}
where
\begin{equation}
  \Phi_m^+(\phi)=\cos (m\phi), \quad \Phi_m^-(\phi)=\sin (m\phi)
\end{equation} and for $\mu\not=0$,
\begin{equation} \label{eq:prolatepreharmonicsq}
   \harorig{n,m}[\mu](u,v) =  \frac{(n-m)!}{2^n(1/2)_n}\mu^n
  P_{n}^{m}(\cos u) P_{n}^{m}(\cosh v) .
\end{equation}
  Here $P_n^m$ are the associated Legendre functions
of the first kind \cite[Ch. III]{Hobson1931} of degree $n$ and order
$m$,  and
 the (rising) Pochhammer symbol is
$(a)_n=a(a+1)\cdots(a+n-1)$ with $(a)_0=1$ by convention.
To avoid repetition, we state once and for all that
$\harorigsol{n,m}{-}[\mu]$ is only defined for $m\ge1$, i.e.\
$\harorigsol{n,0}{-}[\mu]$ is expressly excluded from all statements
of theorems.

It was shown in \cite{GMP} that with the scale factor which has been
included in \eqref{eq:prolatepreharmonicsq}, the
$\harorigsol{n,m}{\pm}[\mu]$ are polynomials in $(x_0,x_1,x_2)$ which
are normalized so that the limiting case $\mu\to0$ gives the classical
{\it solid spherical harmonics},
\begin{equation*}
\harorigsol{n,m}{\pm}[0](x) = |x|^nP_n^m(x_0/|x|) \Phi_m^\pm(\phi).
\end{equation*}
It is known from \cite{Garabedian} that while the
$\harorigsol{n,m}{\pm}[\mu]$ are mutually orthogonal with respect to
the Dirichlet norm on $\Omega_\mu$, the closely related functions,
which we will call the \textit{Garabedian spheroidal harmonics},
\begin{equation}  \label{eq:prolateharmonics}
  \hargarabsol{n,m}{\pm}[\mu](x) =
   \frac{\partial}{\partial x_0} \harorigsol{n+1,m}{\pm}[\mu](x)
\end{equation}
form an orthogonal basis for $\mathcal{H}_2(\Omega_\mu)$, the linear
subspace of real-valued harmonic functions in $L_2(\Omega_\mu)$. This
property makes the $\hargarabsol{n,m}{\pm}[\mu]$ of greater interest
for many considerations. The corresponding boundary Garabedian
harmonics $\hargarab{n,m}[\mu]$ in $\Omega_\mu$ are characterized by
the relation
\begin{equation}  \label{eq:Vhat}
  \hargarabsol{n,m}{\pm}[\mu](x) =  \hargarab{n,m}[\mu](u,v) \,\Phi_m^\pm(\phi).
\end{equation}
We recall \cite{MoraisGuerlebeck2012}
that for spherical harmonics, there is a   formula analogous
to Appell differentiation of monomials,
\begin{align}  \label{eq:V[0]}
   \frac{\partial}{\partial x_0} \harorigsol{n+1,m}{\pm}[0](x)
 =  (n+m+1) \harorigsol{n,m}{\pm}[0](x).
\end{align}
However, $\hargarabsol{n,m}{\pm}[\mu]$ is not so simply related to
$\harorigsol{n,m}{\pm}[\mu]$ for $\mu\not=0$, as was explained in
\cite{Morais4}. We examine such relations in the next section.


\section{Conversions among orthogonal spheroidal
harmonics and spherical  harmonics}

\subsection{Garabedian harmonics expressed by classical harmonics}
 
As mentioned in the Introduction, it is of interest to express the
orthogonal basis of harmonic functions for one spheroid $\Omega_\mu$
in terms of those for another spheroid. It is natural to use the
unit ball $\Omega_0$ as a point of reference, which will be the case in
the first results. We begin the calculation of the coefficients for
the relationships among the various classes of harmonic functions by
presenting various known formulas in a uniform manner.  For $n\ge0$,
consider the rational constants
\begin{equation} \label{eq:cUU}
\cUU{n,m,k} =  \frac{ (1/2)_{n-k}\, (n+m-2k+1)_{2k}}{ (-4)^k (1/2)_{n}\, k!  }
\end{equation}
for $0\le m\le n$, $0\le2k\le n$, and let $\cUU{n,m,k}=0$
otherwise. In the present notation, the main result of \cite{BBS} may
be expressed as follows (i.e.\ the factor
$\cscale{m,n} = (n-m)!/(2n-1)!!$ has been incorporated into
\eqref{eq:cUU}).
\begin{prop}[\cite{BBS}] \label{prop:BBS}
Let $n\ge0$ and $0\le m\le n$. Then
\begin{equation*}
  \harorig{n,m}[\mu] = \sum_{0\le2k\le n-m} \cUU{n,m,k} \mu^{2k}\,
    \harorig{n-2k,m}[0].
\end{equation*}
\end{prop}

An important characteristic of this relation is that the same
coefficients $\cUU{n,m,k}$ work for the ``$+$'' and ``$-$'' cases
(cosines and sines) and, strikingly, for all values of $\mu$.
By \eqref{eq:prolatepreharmonics}, an
equivalent form of expressing Proposition \ref{prop:BBS} is
\begin{equation} \label{eq:UqfromUq}
 \harorigsol{n,m}{\pm}[\mu] = \sum_{0\le2k\le n-m}
     \cUU{n,m,k} \mu^{2k}\, \harorigsol{n-2k,m}{\pm}[0].
\end{equation}

Since $\partial/\partial x_0$ in \eqref{eq:prolateharmonics} is a
linear operator, \eqref{eq:UqfromUq} gives automatically the
corresponding result for the Garabedian harmonics,
 \begin{align} \label{eq:VV}
 \hargarabsol{n,m}{\pm}[\mu] &=
   \sum_{0\le2k\le n-m+1} \cVV{n,m,k} \mu^{2k} \,\hargarabsol{n-2k,m}{\pm}[0],   
\end{align}
where $\cVV{n,m,k}=\cUU{n+1,m,k}$.
This in turn gives via \eqref {eq:V[0]} the following expression in
terms of the spherical harmonics:
\begin{coro} \label{coro:VfromU}
Let $n\ge0$ and $0\le m\le n$. Then
\begin{align*}
  \hargarab{n,m}[\mu]  
   &= \sum_{0\le2k\le n-m+1} \cUV{n,m,k}  \mu^{2k}\, \harorig{n-2k,m}[0] ,
\end{align*}
where \[\cUV{n,m,k}=(n+m-2k+1)\cVV{n,m,k}. \]
 \end{coro}

The coefficients
\begin{align*}    
 \cVUmu{n,m,k} 
    &= \frac{ (n+m+1)! \,(1/2)_{n-2k+1}}{4^k(n+m-2k)!(1/2)_{n+1}}
\end{align*}
 give a similar expression for the Garabedian
basic harmonics $\hargarabsol{n,m}{\pm}[\mu]$ in terms of the standard
harmonics $\harorigsol{n,m}{\pm}[\mu]$ for the same spheroid, rather than
in terms of $\harorig{n,m}[0]$:

\begin{theo} [\cite{Morais4}] \label{th:VfromU}
Let $n\ge0$ and $0\le m\le n$. Then
\begin{equation*}
\DS   \hargarab{n,m}[\mu] = \sum_{0\le 2k\le n-m}
     \cVUmu{n,m,k} \mu^{2k}\, \harorig{n-2k,m}[\mu].
\end{equation*}
\end{theo}


 In \cite{BBS} the inverse relation of \eqref{eq:UqfromUq} was also
derived, expressing $\harorigsol{n,m}{\pm}[0]$ in terms of
$\harorigsol{n,m}{\pm}[\mu]$, via
\begin{equation}  \label{eq:UbacktoU}
 \harorig{n,m}[0] = \sum_{0\le k\le n-m}
     \cUU{n,m,k}^0 \mu^{2k} \, \harorig{n-2k,m}[\mu],
\end{equation}
where the coefficients can be written as
\begin{equation}  \label{eq:UbacktoUcoef}
  \cUU{n,m,k}^0 =
 \frac{ 4^{n-2k}(2n-4k+1)(n-k)!(m+n)!(1/2)_{n-2k}}
      { k!(2n-2k+1)!(n+m-2k)!},
\end{equation}
again independent of $\mu$.  In consequence, applying the operator
$\partial/\partial x_0$ and using \eqref{eq:V[0]}, we have the
following result.
\begin{prop}\label{prop:UfromV}
Let $n\geq0$ and $0\leq m\leq n$. Then
\[   \harorig{n,m}[0] =
    \sum_{0\leq2k\leq n-m}\cUV{n,m,k}^0\mu^{2k} \,\hargarab{n-2k,m}[\mu],
\]
where
\begin{equation*}
 \cUV{n,m,k}^0 = \dfrac{\cUU{n+1,m,k}^0}{n+m+1}.
\end{equation*}
\end{prop}

 The inverse relation for Theorem \ref{th:VfromU} is a much simpler
formula, given as follows:
 
\begin{coro} [\cite{Morais4}]
For  $n\geq0$ and $0\leq m \leq n$,
\[  \harorig{n,m}[\mu] = \frac{1}{n+m+1} \hargarab{n,m}[\mu]
      + \frac{n+m}{4n^2-1}  \mu^2  \,\hargarab{n-2,m}[\mu].
\]
\end{coro}
This uses the convention $\hargarab{n-2,m}[\mu]=0$ when $m>n$;
i.e.
\begin{align*}
  \harorig{n,n-1}[\mu] &= \frac{1}{2n} \hargarab{n,n-1}[\mu],\\
  \harorig{n,n}[\mu] &= \frac{1}{2n+1} \hargarab{n,n}[\mu].
\end{align*}


\subsection{Conversion  among Garabedian harmonics}

The preceding subsection does not include the inverse relation of
\eqref{eq:VV} of the form
\begin{equation}  \label{eq:VVinv} 
  \hargarab{n,m}[0]  = \sum_{0\leq2k\leq n-m} \cVV{n,m,k}^0 \mu^{2k}\hargarab{n-2k,m}[\mu].
\end{equation}
Instead of deriving it directly, we verify first the following
remarkable conversion formula, which relates the spheroidal harmonics
associated with $\Omega_{\mu}$ to those associated with any other
$\Omega_{\mut}$.  Write
\begin{equation*}
\cmumut{n,m,k}=
   \dfrac{(n+m+1)!(1/2)_{n-2k+2}}{4^{k}k!(n+m-2k+1)! (1/2)_{n-k+2}}
\end{equation*}
when $0 \leq 2k \leq n-m+2$,  otherwise $\cmumut{n,m,k}=0$.
  
\begin{theo}\label{th:cVVmu}
  Let $n\geq0$, $0\leq m\leq n$, and let
  $\mu,\mut \in [0,1) \cup i\R^{+}$ such that $\mu\neq0$. The
   coefficients $\cVVmu{n,m,k}[\mut,\mu]$ in the
  relation
\[ \hargarab{n,m}[\mut] = \sum_{0\leq2k\leq n-m}
    \cVVmu{n,m,k}[\mut,\mu]\,\hargarab{n-2k,m}[\mu]
\]
are given by
\[  \cVVmu{n,m,k}[\mut,\mu] =  {}_2F_1(-k,-n+k-3/2;-n-1/2;(\mut/\mu)^2)
   \, \cmumut{n,m,k}\,\mu^{2k},
\]
with $_2F_1$ denoting the classical Gaussian hypergeometric function.
\end{theo}

\begin{proof}
  We begin by replacing $\mu$ with $\mut$ in Corollary
  \ref{coro:VfromU} and substituting the terms on the right-hand side
  according to Proposition \ref{prop:UfromV}.  By linear independence
  of the harmonic basis elements, it follows that
\begin{equation}\label{eq:sumw}
  \cVVmu{n,m,k}[\mut,\mu] = \mu^{2k}\sum_{l=0}^{k}\cUV{n,m,l}\cUV{n-2l,m,k-l}^0  
 \left(\frac{\mut}{\mu}\right)^{2l}
\end{equation}
in which we note that all terms are real valued.  Using reductions such as
$(2n-4k+3)(1/2)_{n-2k+1}=2(1/2)_{n-2k+2}$ and recalling $0\le l\le k$,
one easily sees that
\begin{align*}
  \cUV{n,m,l} &= \frac{ (1/2)_{n-l+1}(n+m-2l+1)_{2l+1}} {(-4^l)l!(1/2)_{n+1}}, \\
   \cUV{n-2l,m,k-l}^0 &= \frac{2\cdot4^{n-2k+1}(n+m-2l)!(n-k-l+1)!(1/2)_{n-2k+2} }
         { (k-l)! (2n-2k-2l+3)! (n+m-2k+1)! }.
\end{align*}
Therefore the product can be expressed as
\[  \cUV{n,m,l} \cUV{n-2l,m,k-l}^0 = \cmumut{n,m,k} \chypgeom{n,k,l} 
\]
where
\begin{align*}
  \chypgeom{n,k,l} &=  \frac{ 2\cdot 4^{n-2k+1}(n+m+1)!(n-k-l+1)!(1/2)_{n-2k+2} } 
                  { (-4^l)l!(k-l)! (2n-2k-2l+3)! (n+m-2k+1)! } \\
  &= \frac{ (-k)_l(-n+k-3/2)_l }{ l!(-n-1/2)_l}
\end{align*}
is the coefficient in the polynomial
$  {}_2F_1(-k,-n+k-3/2;-n-1/2;(\mut/\mu)^2) =
 \sum_{l=0}^k \chypgeom{n,k,l} (\mut/\mu)^{2l}$. 
\end{proof}
 
\begin{coro}\label{coro:proplim}
For each $n\geq0$, $0\leq m\leq n$, the limits
\begin{equation*}
 \lim_{\mut\rightarrow0} \cVVmu{n,m,k}[\mut,\mu], \quad
 \lim_{\mu\rightarrow0} \cVVmu{n,m,k}[\mut,\mu]
\end{equation*}
exist and are given, respectively, by  
\begin{equation*}
 \cVVmu{n,m,k}[0,\mu] = (n+m+1)\cUV{n,m,k}^0\mu^{2k}, \quad
 \cVVmu{n,m,k}[\mut,0] = \dfrac{\cUV{n,m,k}}{n+m-2k+1}\mut^{2k}.
\end{equation*}
\end{coro}

\begin{proof}
We may write \eqref{eq:sumw} as
\begin{align*}
 \cVVmu{n,m,k}[\mut,\mu] &=
   \sum_{l=1}^{k-1} \cUV{n,m,l} \cUV{n-2l,m,k-l}^0  \, \mu^{2(k-l)}\mut^{2l} \\
 & \quad\ \ +\cUV{n,m,k}\cUV{n-2k,m,0}^0\mut^{2k}
   + \cUV{n,m,0}\cUV{n,m,k}^0\mu^{2k} 
\end{align*}
and then simply take $\mu=0$ or $\mut=0$ to obtain  the desired limit.
\end{proof}

Referring to \eqref{eq:VVinv}, we have
\[ \cVV{m,n,k}^0=  \frac{\cUV{n,m,k}}{(n+m-2k+1)}. \]


\section{Application to orthogonal monogenic 
and contragenic functions}

The standard bases for spheroidal harmonics have their counterparts
for the spaces of orthogonal monogenic polynomials taking values in
$\mathbb{R}^3$. Monogenic functions are defined by considering
$\mathbb{R}^3$ as the real linear subspace of the quaternions
$\mathbb{H} = \{\sum_{i=0}^3 x_ie_i\colon\ x_i \in \mathbb{R}\}$ for
which the last coordinate $x_3$ vanishes. (Quaternionic multiplication
is defined, as usual, so that $e_1^2 = e_2^2 = e_3^2 = -1$ and
$e_1 e_2 = e_3 = - e_2 e_1$, $e_2 e_3 = e_1 = - e_3 e_2$,
$e_3 e_1 = e_2 = - e_1 e_3$.) For background on quaternionic analysis
in $\R^3$, see \cite{Delanghe2007,Joao2009,
  MoraisGuerlebeck2012,MoraisGuerlebeck22012,MoraisAG}. A function
$f\colon\Omega_\mu\to\R^3$ is \textit{monogenic} when it is
annihilated by the quaternionic differential operator
$ \partial = \partial/\partial x_0+e_1\partial/\partial
x_1+e_2\partial/\partial x_2$ acting from the left. The \textit{basic
  spheroidal monogenic polynomials} are constructed
\cite{Morais4,Morais5} as
\begin{equation} \label{def:monog}
  \monog{n,m}{\pm}[\mu]  = \overline{\partial}(\harorigsol{n+1,m}{\pm}[\mu]),
\end{equation}
where
$ \overline{\partial} = \partial/\partial x_0 - e_1\partial/\partial
x_1- e_2\partial/\partial x_2$.
This is analogous to the definition \eqref{eq:prolateharmonics} for
harmonic polynomials. $\monog{n,m}{\pm}[\mu]$ is monogenic because
$\partial\overline{\partial}$ is equal to the Laplacian operator. We continue with the convention that $m\ge1$ when the ``-'' sign appears
in a superscript.

\begin{theo} [\cite{Morais4,Morais5}] \label{th:sphmonogformula}
  For all $n\geq0$, the basic spheroidal monogenic polynomial
  \eqref{def:monog} is equal to
\begin{align*}
  \monog{n,0}{+}[\mu] = \hargarabsol{n,0}{+}[\mu] -
   \frac{1}{n+2} \big(  \hargarabsol{n,1}{+}[\mu] e_1 +
                        \hargarabsol{n,1}{-}[\mu] e_2  \big)
\end{align*}
for $m=0$, and 
\begin{align*}
 \monog{n,m}{\pm}[\mu]  &=
 \hargarabsol{n,m}{\pm}[\mu] + \Bigl[(n+m+1) \hargarabsol{n,m-1}{\pm}[\mu] -
  \frac{1}{n+m+2} \hargarabsol{n,m+1}{\pm}[\mu] \Bigr]\frac{e_1}{2} \nonumber\\
 & \phantom{\quad \hargarabsol{n,m}{\pm}[\mu]\ } \mp \Bigl[(n+m+1) \hargarabsol{n,m-1}{\mp}[\mu] +
 \frac{1}{n+m+2} \hargarabsol{n,m+1}{\mp}[\mu] \Bigr] \frac{e_2}{2} 
\end{align*}
for   $1\leq m\leq n+1$.
  The polynomials $\monog{n,m}{\pm}[\mu]$
 are orthogonal in $L^2(\Omega_\mu)$, i.e.\ in the sense of the scalar
 product defined by
\[   \langle f,g\rangle_{[\mu]} = \int_{\Omega_\mu} {\rm Sc}(\overline{f}g)\,dV.
\]
\end{theo}


\subsection{Bases for monogenics in distinct spheroids}

Analogously to \eqref{eq:VV} and \eqref{eq:VVinv},
we now express $\monog{n,m}{\pm}[\mu]$ in terms of the spherical
monogenics $\monog{n,m}{\pm}[0]$.

\begin{theo}\label{th:monogenics} 
  For  $n\geq0$ and $0\leq m \leq n+1$,
\begin{align*}
  \monog{n,m}{\pm}[\mu] =& \sum_{0\le 2k\le n-m+1}
      \cVV{n,m,k} \mu^{2k} \monog{n-2k,m}{\pm}[0],\\
  \monog{n,m}{\pm}[0]   =& \sum_{0\le 2k\le n-m+1}
     \cVV{n,m,k}^0 \mu^{2k} \monog{n-2k,m}{\pm}[\mu],\\
  \monog{n,m}{\pm}[\mut]   =& \sum_{0\leq2k\leq n-m+1}
         \cVV{n,m,k}[\mut,\mu]\, \monog{n-2k,m}{\pm}[\mu],
\end{align*}
where $\cVV{n,m,k}$,  $\cVV{n,m,k}^0$, and $\cVV{n,m,k}[\mu,\mut]$   are
  as in the previous section.
\end{theo}

\begin{proof} Fix a value of $\mu$. Note that for given $n$, the
  collections $\{\DS\monog{k,m}{\pm}[0]\colon\ k\le n,\ 0\le m\le k\}$ and
  $\{\DS\monog{k,m}{\pm}[\mu]\colon\ k\le n,\ 0\le m\le k\}$ are bases for the same
  linear space, namely the monogenic $\mathbb{R}^3$-valued polynomials in the
  variables $(x_0,x_1,x_2)$ of degree $\le n$. Therefore there must
  exist real coefficients $a^\pm_{k}$ such that
  $\monog{n,m}{+}[\mu]=\sum_k\sum_m a^+_k\monog{n,k}{+}[0] +\sum_k\sum_m 
  a^-_k\monog{n,k}{-}[0]$.
  By Theorem \ref{th:sphmonogformula}, the scalar part of this equation
  expresses the spheroidal harmonics $\hargarabsol{n,m}{\pm}[\mu]$ as a
  linear combination of the spherical harmonics
  $\hargarabsol{k,m}{\pm}[0]$. By the uniqueness of the representation
  \eqref{eq:VV} we have that $a^\pm_{k} = \cVV{n,m,k} \mu^{2k}$. The
  second formula follows by the same reasoning, and then the relationship
  between $\monog{n,m}{\pm}[\mu]$ and $\monog{n,m}{\pm}[\mut]$ is a
  consequence of the fact that by Theorem \ref{th:cVVmu} the matrix
  $(\cVV{n,m,k}[\mut,\mu])_{n,k}$ is essentially the product of
  $(\cVV{n,m,k}\mut^{2k})_{n,k}$ and the inverse of
  $(\cVV{n,m,k}^0\mu^{2k})_{n,k}$.
\end{proof}

\subsubsection{Spheroidal ambigenic polynomials} \textit{Antimonogenic} 
functions (quaternionic conjugates of monogenics, i.e.\ annihilated by
$\overline{\partial}$) are generally not studied independently, since
their properties may be obtained by taking the conjugate of facts
about monogenic functions.  For example, the basic antimonogenic
polynomials satisfy essentially the same relation as given in Theorem
\ref{th:monogenics},
\[ \antimonog{n,m}{\pm}[\mu] =
  \sum_{0\leq2k\leq n-m} \cVV{n,m,k}[\mu,\mut]\, \antimonog{n-2k,m}{\pm}[\mut].
\]
However, the subspace of the $\R^3$-valued harmonic functions
generated by the monogenic and antimonogenic functions together, that
is, the \textit{ambigenic} functions \cite{Alvarez}, is of interest.

An ambigenic function is not represented uniquely as a sum of a
monogenic and an antimonogenic function because one may add and
subtract a \textit{monogenic constant}, that is, a function which
is simultaneously monogenic and antimonogenic.
A collection of ambigenic polynomials denoted
$\{\ambibasic{n,m}{\pm,\pm}[\mu]\}$ was constructed in \cite{GMP} and
shown to be a basis of $2n(n+3)+3$ elements for the ambigenic
polynomials of degree no greater than $n$, mutually orthogonal in
$L^2(\Omega_\mu)$. For our purposes we will only need  the particular
ambigenic functions
\begin{equation}  \label{eq:defambi}
  \ambig{n,m}{\pm}[\mu] = 2\Vec \monog{n,m}{\pm}[\mu]
  = \monog{n,m}{\pm}[\mu] - \antimonog{n,m}{\pm}[\mu],
\end{equation}  
where $q=\Sc q + \Vec q$ denotes the decomposition of a quaternionic
quantity into its scalar and vector parts. It is simple to verify
that for fixed $\mu$, the $\ambig{n,m}{\pm}[\mu]$ are linearly independent.


\subsection{Relations among contragenic functions for distinct spheroids}

The notion of contragenic harmonic functions was introduced in
\cite{Alvarez}, arising from the previously unobserved fact that in
contrast to $\C$-valued or $\H$-valued functions, there exist
$\R^3$-valued harmonic functions which are not ambigenic.  Thus a
function is called \textit{contragenic} for a given domain $\Omega$
when it is orthogonal in $L^2(\Omega)$ to all monogenic and
antimonogenic functions in $\Omega$. In contrast to monogenicity and
antimonogenicity, this is not a local property and therefore cannot be
characterized in general by direct application of any differential
operator.  It is of interest to have a basis for the contragenic
functions, in order to express an arbitrary harmonic function in a
calculable way as a sum of an ambigenic function and a contragenic
function. In the following, we will write
\begin{align*}
  \N_*^{(n)}[\mu] =\ & 
  \{\parbox[t]{.6\textwidth}{polynomials 
     of degree $\le n$ in $x_0,x_1,x_1$ which are
     orthogonal in $L_2(\Omega_\mu)$ to all ambigenic
     functions in  $\Omega_\mu\}$,}
\end{align*}
for $n\ge 1$ (nonzero constant harmonic functions are never
contragenic, so we will have no use for $\N_*^{(0)}[\mu]=\{0\}$), and
we have the successive orthogonal complements
\[ \N^{(n)}[\mu] = \N_*^{(n)}[\mu] \ominus \N_*^{(n-1)}[\mu], \]
which are composed of polynomials of degree precisely $n$. Thus
$\N_*^{(n)}[\mu] =\bigoplus_{k=1}^n\N^{(k)}[\mu]$ and there is a
Hilbert space orthogonal decomposition
$\N_*[\mu] =\bigoplus_{k=1}^\infty\N^{(k)}[\mu]$ of the full
collection of contragenic functions in $L^2(\Omega_\mu)$. The
following explicit construction of a basis of the $\N^{(n)}[\mu]$,
using as building blocks the scalar components of the monogenic
functions, can be found in \cite{GMP}. Write 
\begin{align} 
 \normrat{n,0}[\mu]=& \ 1, \nonumber \\ 
 \normrat{n,m}[\mu] =& \left( \frac{1}{(n+m+1)_2}
     \frac{ \|\hargarabsol{n,m+1}{+}[\mu]\|_{[\mu]}}
          {  \|\hargarabsol{n,m-1}{+}[\mu] \|_{[\mu]}}\right)^2   \label{eq:normrat} 
\end{align}
for $1\le m\le n-1$, and $\normrat{n,m}[\mu]=0$ for $m\ge n$ since
then $\hargarabsol{n,m}{\pm}[\mu]=0$ (this definition involves a slight modification of the
notation in \cite{GMP}), where integration over the ellipsoid gives
explicitly
\begin{equation*}\label{eq:harqnorms}
    \|\hargarabsol{n,m}{\pm}[\mu]\|_{[\mu]}^2 =
     (1+\delta_{0,m})\normconst{n,m} \pi \mu^{2n+3} 
     \int_{1}^{\frac{1}{\mu}}P_{n}^{m}(t)P_{n+2}^{m}(t)\,dt. 
\end{equation*}
Here $\delta_{m,m'}$ is the Kronecker symbol and
\[   \normconst{n,m}  =
    \frac{  (n+m+1) (n+m+1)!(n-m+2)!} {2^{2n+1} (1/2)_{n+1}(1/2)_{n+2} }  .
\]

\begin{defi}  \label{def:Basic_contragenics} For all $n\ge1$, the
\textit{basic contragenic polynomials} $\contra{n,m}{\pm}[\mu]$
associated to $\Omega_\mu$ are
\begin{align*}
   \contra{n,0}{+}[\mu] =& -\ambig{n,0}{+}[\mu] e_3  
\end{align*}
for $m=0$, and
\begin{align*}
 \contra{n,m}{\pm}[\mu] =
 \frac{1}{2}\big(
 \mp(\normrat{n,m}[\mu]+1)\ambig{n,m}{\pm}[\mu]
  + (\normrat{n,m}[\mu]-1)\ambig{n,m}{\mp}[\mu] e_3 \big)
 \end{align*} 
for $1\leq m\leq n-1$, where $\ambig{n,m}{\pm}[\mu]$ are defined by \eqref{eq:defambi}.
\end{defi}

In \cite{GMP} it was shown that
$\{\contra{n,m}{\pm}[\mu]\colon\ 0\le m< n-1\}$ is an orthonormal
basis for $\N^{(n)}[\mu]$, and that the harmonic polynomials of degree
$\le n$ in $\Omega_\mu$ decompose as orthogonal direct sums of the
ambigenic and contragenic polynomials of degree $\le n$.  With the
further notation
\begin{align*}
  \angPsi{+,m}{\pm} &= \Phi_m^{\pm}(\phi)e_1 \pm
             \Phi_m^{\mp}(\phi)e_2, \\
  \angPsi{-,m}{\pm} &= \Phi_m^{\pm}(\phi)e_1 \mp
             \Phi_m^{\mp}(\phi)e_2,
\end{align*}
which satisfy the obvious relations
$\angPsi{+,m}{\pm}e_3=\pm\angPsi{+,m}{\mp}$,
$\angPsi{-,m}{\pm}e_3=\mp\angPsi{-,m}{\mp}$,
$e_1\hargarabsol{n,m}{\pm}[\mu]+e_2\hargarabsol{n,m}{\mp}[\mu]=
\hargarab{n,m}\angPsi{\pm,m}{\pm}[\mu]$,
$e_1\hargarabsol{n,m}{\pm}[\mu]-e_2\hargarabsol{n,m}{\mp}[\mu]=
\hargarab{n,m}\angPsi{\mp,m}{\pm}[\mu]$  (where the
$\hargarab{n,m}[\mu]$ are given by \eqref{eq:Vhat}), the definitions
give us almost immediately that
\begin{align}
 \ambig{n,0}{+}[\mu] &=
   \frac{-2}{n+2} \hargarab{n,1}[\mu] \angPsi{+,1}{+},\nonumber\\
 \ambig{n,m}{\pm}[\mu] &=  
   (n+m+1)\hargarab{n,m-1}[\mu] \angPsi{-,m-1}{\pm} \nonumber\\
     &\ \ -\dfrac{1}{n+m+2}\hargarab{n,m+1}[\mu] \angPsi{+,m+1}{\pm}, \label{eq:ambiformula}  
\end{align}
\begin{align}
 \contra{n,0}{+}[\mu] &=
   \frac{2}{n+2} \hargarab{n,1}[\mu]\angPsi{+,1}{-},\nonumber\\
 \contra{n,m}{\pm}[\mu] &= (n+m+1)\normrat{n,m}[\mu]
   \hargarab{n,m-1}[\mu] \angPsi{-,m-1}{\mp} \nonumber \\
     &\ \ + \dfrac{1}{n+m+2}\hargarab{n,m+1}[\mu]
       \angPsi{+,m+1}{\mp},   \label{eq:contraformula}
\end{align}
where $1 \le m \leq n-1$. 

Adding and subtracting instances of \eqref{eq:ambiformula} and
\eqref{eq:contraformula} gives by cancellation  decompositions of the
harmonic polynomials $\hargarab{n,m}\angPsi{+,m}{\pm}$ and $\hargarab{n,m}\angPsi{-,m}{\pm}$ as the sum of a
contragenic and an ambigenic:

\begin{lema}\label{lem:VZA}
Let $n\geq1$ and  $1\leq m \leq n+1$. Then
\begin{align*}
  \hargarab{n,m-1}[\mu]\angPsi{-,m-1}{\pm} =
   \dfrac{1}{(n+m+1)(\normrat{n,m}[\mu]+1)}\big(&\contra{n,m}{\mp}[\mu]
     + \ambig{n,m}{\pm}[\mu] \big),
\end{align*}
and
\begin{align*}
  \hargarab{n,m+1}[\mu]\angPsi{+,m+1}{\pm} =
   \dfrac{n+m+2}{ \normrat{n,m}[\mu]+1}\big(&\contra{n,m}{\mp}[\mu]
     - \normrat{n,m}[\mu]\ambig{n,m}{\pm}[\mu] \big) .
\end{align*}
\end{lema}

The definition of contragenic function does not imply that an
$L^2$-function which belongs to the space $\N_*^{(n)}[\mut]$ should also
be in $\N_*^{(n)}[\mu]$ when $\mut\neq\mu$, because the notion of
orthogonality is different for different spheroids. In other words, we
may not expect a formula like
``$\contra{n,m}{\pm}[\mut]=\sum
z_{n,m,k}[\mut,\mu]\contra{n-2k,m}{\pm}[\mu]$.'' The following result
will enable us to give many examples for which
$\contra{n,m}{\pm}[\mut]\not\in\N_*^{(n)}[\mu]$ for $m\geq1$. However,
it also shows that the intersection of all of the $\N_*^{(n)}[\mu]$ is
nontrivial, giving what may be called {\it universal contragenic
  functions} in the context of spheroids.

We will use the coefficients
\begin{align}
 \cZZmu{n,0,k}{C}[\mut,\mu]  =& \
     \dfrac{n-2k+2}{n+2}\cVVmu{n,1,k}[\mut,\mu], \nonumber\\
 \cZZmu{n,m,k}{C}[\mut,\mu] =&\   \left\{\begin{array}{ll}
     \dfrac{\normrat{n,m}[\mut]+1}{\normrat{n-2k,m}[\mu]+1}
       \cVVmu{n,m,k}[\mut,\mu],\ \quad & 0\le2k\le n-m-1,\\[2ex]
     \dfrac{\normrat{n,m}[\mut]}{\normrat{n-2k,m}[\mu]+1}
       \cVVmu{n,m,k}[\mut,\mu], &   n-m\le2k\le n-m+1; 
          \end{array}\right. \nonumber \\
   \cZZmu{n,m,k}{A}[\mut,\mu] =&\ \left\{\begin{array}{ll}
      \dfrac{\normrat{n,m}[\mut]-\normrat{n,m}[\mu]}{\normrat{n-2k,m}[\mu]+1}
       \cVVmu{n,m,k}[\mut,\mu],  & 0\le2k\le n-m-1,\\[2ex]
        \dfrac{\normrat{n,m}[\mut] }{\normrat{n-2k,m}[\mu]+1}
       \cVVmu{n,m,k}[\mut,\mu], &  n-m\le2k\le n-m+1; 
               \end{array}\right.     \label{eq:CZZmu}
\end{align} 
($1 \le m \le n-1$) to express the decomposition of contragenics for one
spheroid in terms of contragenics and ambigenics of any other.

\begin{prop} \label{prop:contragenicrelations}
   Let $n\geq1$. Then
\begin{align*}
  \contra{n,0}{+}[\mut] &= \sum_{0\leq 2k\leq n-1} 
   \cZZmu{n,k}{C}[\mut,\mu]\contra{n-2k,0}{}[\mu];
\end{align*}
 and for $1\leq m\leq n-1$,
\begin{equation*}
   \contra{n,m}{\pm}[\mut]=\sum_{0\leq 2k\leq n-m+1}
   \big( \cZZmu{n,m,k}{C}[\mut,\mu]\contra{n-2k,m}{\pm}[\mu] +
         \cZZmu{n,m,k}{A}[\mut,\mu] \ambig{n-2k,m}{\pm}[\mu] \big).
\end{equation*}
 \end{prop}

\begin{proof}
  Apply Theorem \ref{th:cVVmu} to the first formula of
  \eqref{eq:contraformula} with $\mut$ in place of $\mu$ to obtain that
\begin{align*}
 \contra{n,0}{+}[\mut] =  \frac{2}{n+2} \sum_{0\leq 2k\leq n-1}
     \cVVmu{n,1,k}[\mut,\mu]\hargarab{n-2k,1}[\mu] \angPsi{+,1}{-}  ,
\end{align*}
which after another application of \eqref{eq:contraformula} reduces to
the first statement.  In the same way, for $m\ge1$,
\begin{align}
  \contra{n,m}{\pm}[\mut] =&
   \  (n+m+1)\normrat{n,m}[\mut] \sum_{0\leq 2k\leq n-m+1}
      \cVVmu{n,m-1,k}[\mut,\mu]\hargarab{n-2k,m-1}[\mu] 
       \angPsi{-,m-1}{\pm} \nonumber\\
   & \  + \dfrac{1}{n+m+2} \sum_{0\leq 2k\leq n-m-1}
       \cVVmu{n,m+1,k}[\mut,\mu]\hargarab{n-2k,m+1}[\mu] 
       \angPsi{+,m+1}{\pm}. \label{eq:Zsum}
\end{align}
We observe from the definitions leading to Proposition
\ref{prop:UfromV} that
\begin{align*}
  \cUV{n,m-1,l}\,\cUV{n-2l,m-1,k-l}^0
 = \dfrac{n+m-2k+1}{n+m+1}\cUV{n,m,l}\,\cUV{n-2l,m,k-l}^0,
\end{align*}
so \eqref{eq:sumw} tells us that
\begin{align*}
\dfrac{n+m+1}{n+m-2k+1}\cVVmu{n,m-1,k}[\mut,\mu] &= \cVVmu{n,m,k}[\mut,\mu] \\
&= \dfrac{n+m-2k+2}{n+m+2}\cVVmu{n,m+1,k}[\mut,\mu].
\end{align*}
From this and Lemma \ref{lem:VZA} we have that
\begin{align*}
   (n+m+1)      \cVVmu{n,m-1,k}[\mut,\mu] \hargarab{n-2k,m-1}[\mu] 
      & \angPsi{-,m-1}{\pm} \\
  = \frac{1}{\normrat{n-2k,m}[\mu]+1} \cVV{n,m,k}[\mut,\mu]
   (\contra{n-2k,m}{\mp}&[\mu]+\ambig{n-2k,m}{\pm}[\mu]  ),
\end{align*}
and
\begin{align*}
  \frac{1}{n+m+2}\cVVmu{n,m+1,k}[\mut,\mu] \hargarab{n-2k,m+1}[\mu]
      & \angPsi{+,m+1}{\pm} \\
  = \frac{1}{\normrat{n-2k,m}[\mu]+1} \cVV{n,m,k}[\mut,\mu]
    (\contra{n-2k,m}{\mp}& [\mu]- \normrat{n-2k,m}[\mu]\ambig{n-2k,m}{\pm}[\mu]  ).
\end{align*}
Inserting these two relations into the respective sums of
\eqref{eq:Zsum} gives the desired result.
\end{proof}

Proposition \ref{prop:contragenicrelations} provides us with some
information about the intersection of the spaces of contragenic
functions up to degree $n$.

\begin{theo}\label{th:intersection}
  Let $n\geq1$. The following statements hold:
\begin{enumerate}
\item[(i)]  $\contra{n,0}{+}[\mu]\in\N_*^{(n)}[0]$ for all $\mu$;
\item[(ii)] 
  $\contra{n,m}{\pm}[\mu] \notin N_*^{(n)}[0]$
   when $\mu\neq0$ and $1\le m\le n-1$.
\end{enumerate}
\end{theo}

\begin{proof} The first statement is an immediate consequence of the
  first formula of Proposition \ref{prop:contragenicrelations}.

  Now consider a basic element $\contra{n,m}{\pm}[\mu]$ of
  $\N_*^{(n)}[\mu]$, with $\mu\not=0$ and $1\leq m\leq n-1$. A
  particular instance of the second formula of Proposition
  \ref{prop:contragenicrelations} is
\[  \contra{n,m}{\pm}[\mu]=\sum_{0\leq 2k\leq n-m+1}
   \big( \cZZmu{n,m,k}{C}[\mu,0] \contra{n-2k,m}{\pm}[0] +
   \cZZmu{n,m,k}{A}[\mu,0]\ambig{n-2k,m}{\pm}[0] \big).
\]
Suppose that $\contra{n,m}{\pm}[\mu]\in\N_*^{(n)}[0]$. Then since
the right hand side is orthogonal to all $\Omega_0$-ambigenics,
\[ \sum_{0\leq 2k\leq n-m+1}
  \cZZmu{n,m,k}{A}[\mu,0]\ambig{n-2k,m}{\pm}[0]=0, 
\]
and so by the linear independence, $\cZZmu{n,m,k}{A}[\mu,0]=0$ for all
$k$. The case in \eqref{eq:CZZmu} where $2k$ is  $n-m$ or
$n-m+1$ tells us that  $\normrat{n,m}[\mu]=0$, which is
manifestly false by \eqref{eq:normrat}. Consequently,
$\contra{n,m}{\pm}[\mu]\not\in\N_*^{(n)}[0]$ as claimed.
\end{proof}

Note that Theorem \ref{th:intersection} does not assert that
$\contra{n,0}{+}[\mu]$ lies in the top-level slice $\N^{(n)}[0]$ of
$\N_*^{(n)}[0]$.

 \begin{coro}
Let $n\geq1$. Then
\[ \dim \bigcap_{\mu\in[0,1)\cup\R^+}\N_*^{(n)}[\mu] \ge n.
\]
\end{coro}

\begin{proof}
  The result is an immediate consequence of  the fact that Theorem
  \ref{th:intersection} is applicable to arbitrary $\mu$, so
the intersection contains a fixed $n$-dimensional subspace of
  $\N_*^{(n)}[0]$.
\end{proof}

It also follows from Theorem \ref{th:intersection} that the common
intersection $\N_0=\bigcap\N_*[\mu]$ of the full spaces of contragenic
functions on spheroids is infinite dimensional, containing all of the
contragenic polynomials $\contra{n,m}{+}[\mu]$ for which $m=0$.  It
seems likely that these contragenic polynomials have special
characteristics because of their simpler structure, cf.\
\eqref{eq:contraformula}. This phenomenon is not yet fully understood. 
Further questions relating to the exact relations among
the spaces $\N_*^{(n)}[\mu]$ still remain open. If the method of the
proof of Theorem \ref{th:intersection} is applied to linear
combinations of the $\contra{n,m}{\pm}[\mu]$ instead of just to these
generators individually, transcendental equations related to
\eqref{eq:normrat} appear. It is not yet known how the angles
between the orthogonal complements of the mode-$0$ subspace
$\N_0^{n}[0]$ in $\N_*^{(n)}[\mu]$, or of their union $\N_0[0]$ in
$\N[\mu]$, vary with $\mu$.


\end{document}